\documentclass[12pt,twoside,reqno]{amsart}
\usepackage{amsmath}
\usepackage{amsfonts}
\usepackage{amssymb}
\usepackage{color}
\usepackage{mathrsfs}
\usepackage{times}
\newcommand{\RR}{\mathbb R}

\newcommand{\beqn}{\begin{equation}}
\newcommand{\eeqn}{\end{equation}}
\newcommand{\bean}{\begin{eqnarray}}
\newcommand{\eean}{\end{eqnarray}}
\DeclareMathAlphabet{\mathpzc}{OT1}{pzc}{m}{it}
\newtheorem{theorem}{Theorem}[section]
\newtheorem{corollary}[theorem]{Corollary}
\newtheorem{lemma}[theorem]{Lemma}
\newtheorem{proposition}[theorem]{Proposition}

\newtheorem{remark}[theorem]{Remark}

\numberwithin{equation}{section}
\begin{document}
\title{Sign-Preserving Property for Some Fourth-Order Elliptic Operators in One Dimension and Radial Symmetry} 

\author{Philippe Lauren\c{c}ot}
\address{Institut de Math\'ematiques de Toulouse, CNRS UMR~5219, Universit\'e de Toulouse \\ F--31062 Toulouse Cedex 9, France}
\email{laurenco@math.univ-toulouse.fr}

\author{Christoph Walker}
\address{Leibniz Universit\"at Hannover\\ Institut f\" ur Angewandte Mathematik \\ Welfengarten 1 \\ D--30167 Hannover\\ Germany}
\email{walker@ifam.uni-hannover.de}


\date{\today}

\begin{abstract}
For a class of one-dimensional linear elliptic fourth-order equations with homogeneous Dirichlet boundary conditions it is shown that a non-positive and non-vanishing right-hand side gives rise to a negative solution. A similar result is obtained for the same class of equations  for radially symmetric solutions in a ball or in an annulus. Several applications are given, including applications to nonlinear equations and eigenvalue problems.
\end{abstract}

\maketitle

%
%
\pagestyle{myheadings}
\markboth{\sc{Ph. Lauren\c{c}ot and Ch. Walker}}{\sc{Sign-Preserving Property for 1D-Fourth-Order Equations}}

\section{Introduction}

A central tool for the analysis of linear and nonlinear second-order elliptic and parabolic equations is the celebrated maximum principle which, roughly speaking, guarantees non-negativity (or even positivity) of solutions provided the boundary and/or initial data are non-negative and the equation satisfies suitable properties. It has far reaching applications, not only to well-posedness issues (for instance, being at the basis of the theory of viscosity solutions, see e.g. \cite{CIL92}), but also to the qualitative behavior of solutions, see \cite{GT01, PS07a, PS07b, QS07} for instance and the references therein. Owing to its powerfulness, a natural question is whether a similar tool is available for higher order linear elliptic operators and, in particular, for the biharmonic operator $\Delta^2$ in a bounded domain $\Omega$. This question has a long and rich history, and we refer to \cite{GGS10, Gr09} for a detailed account and references. Roughly speaking, the validity of the maximum principle for the biharmonic operator $\Delta^2$ in a bounded domain $\Omega$ turns out to depend heavily on the boundary conditions on $\partial\Omega$ and on the domain $\Omega$ itself. For instance, given a non-positive smooth function $f:~\bar{\Omega}\to (-\infty,0]$ with $f\not\equiv 0$, a straightforward application of the maximum principle for the Laplace operator with homogeneous Dirichlet boundary conditions in $\Omega$ reveals that the unique classical solution $u$ to the biharmonic equation with Navier or pinned boundary conditions
\begin{equation}
\Delta^2 u = f \;\text{ in }\; \Omega\ , \qquad u=\Delta u = 0 \;\text{ on }\; \partial\Omega\ , \label{i1}
\end{equation}
is negative in $\Omega$ in the sense that $u(x)<0$ for $x\in\Omega$, a property which will be referred to as the \textsl{strong sign-preserving property} throughout the paper. A similar result, however, fails to be true in general when the Navier boundary conditions are replaced by Dirichlet or clamped boundary conditions, that is, when $u$ solves
\begin{equation}
\Delta^2 u = f \;\text{ in }\; \Omega\ , \qquad u=\partial_n u = 0 \;\text{ on }\; \partial\Omega\ , \label{i2}
\end{equation}
where $\partial_n u$ denotes the normal trace of the gradient of $u$ on the boundary. Nevertheless, it was observed by Boggio  \cite{Bo05} that, when $\Omega$ is the unit ball $\mathbb{B}_1$ of $\mathbb{R}^d$, $d\ge 1$, the Green function associated with \eqref{i2} can be computed explicitly and is positive in $\mathbb{B}_1$. An obvious consequence of this positivity property is that solutions to \eqref{i2} with a non-positive right-hand side $f\not\equiv 0$ are negative. Actually, Boggio's celebrated result provides the impetus for several further studies, in particular the extension of the strong sign-preserving property to other domains $\Omega$ and the analysis of semilinear equations of the form
\begin{equation*}
\Delta^2 u = F(u) \;\text{ in }\; \mathbb{B}_1 , \qquad u=\partial_n u = 0 \;\text{ on }\; \partial \mathbb{B}_1\ ,
\end{equation*}
under suitable assumptions on the nonlinearity $F$, e.g. see \cite{EGG10, GGS10} and the references therein. The strong sign-preserving property of the biharmonic operator with homogeneous Dirichlet boundary conditions in $\mathbb{B}_1$ being at the heart of further investigations, it is therefore tempting to figure out whether it holds true not only for other domains as discussed, for instance, in \cite{GS96, Sa07}, but also for more general fourth-order operators. Besides its theoretical interest, this question also has many applications, e.g. to small deformations of a membrane clamped at its boundary governed by the equation
\begin{equation}
B \Delta^2 u - T \Delta u= f \;\text{ in }\; \Omega , \qquad u=\partial_n u = 0 \;\text{ on }\; \partial\Omega\ , \label{i4}
\end{equation}
where $B \Delta^2 u$ with $B>0$ accounts for bending and $-T \Delta u$ with $T>0$ for stretching. We refer to \cite{EGG10} for a concrete application to microelectromechanical systems (MEMS) and to Theorem~\ref{MEMS} below. In the particular case \eqref{i4}, the strong sign-preserving property has been established in \cite{Ow97} in one space dimension, i.e. when $\Omega=(-1,1)$, the proof relying on the explicit computation of the Green function and its positivity as in \cite{Bo05}. This result is extended in \cite{Gr02}, where the equation
\begin{equation}
u'''' + a u''' + \lambda u''= f \;\text{ in }\; (-1,1)\ , \qquad u(\pm 1) = u'(\pm 1) = 0 \ , \label{i5}
\end{equation}
is shown to enjoy the strong sign-preserving property for $(a,\lambda)\in \RR\times (-\infty,0]$ and $(a,\lambda)\in \{0\}\times (0,\pi^2)$. Interestingly, in contrast to \cite{Bo05, Ow97}, the proof in \cite{Gr02} does not rely on the explicit computation of the Green function associated with \eqref{i5} but on maximum principle arguments applied to second-order elliptic equations after writing \eqref{i5} as an equivalent system of two second-order elliptic equations
\begin{equation*}
u'' =  \gamma\ , \quad \gamma'' + a \gamma' + \lambda \gamma = f \;\text{ in }\; (-1,1)\ , \qquad u(\pm 1) = u'(\pm 1) = 0 \ .
\end{equation*}
Since no boundary condition is given for $\gamma$, the maximum principle cannot be applied directly to the equation for $\gamma$ and a preliminary analysis is required. The possibility of writing \eqref{i5} as a system of two second-order elliptic equations, one with overdetermined boundary conditions and the other with none as in the proof of \cite{Gr02}, is actually also the starting point of the study of the strong sign-preserving property for one-dimensional linear fourth-order operators with general boundary conditions performed in \cite{Sc65, Sc68}. More precisely, the first step in \cite{Sc65} is to transform the fourth-order linear equation
\begin{equation}
 a_4(x) u'''' + a_3(x) u''' + a_2(x) u'' + a_1(x) u' + a_0(x) u = f(x) \label{i6}
\end{equation}
to the system
\begin{equation}
\mathcal{L}_1 u = \gamma\ , \qquad \mathcal{L}_2 \gamma = w(x) \left( f(x) + q(x) u \right) \label{i7}
\end{equation}
for some functions $w>0$ and $q$, which can be computed explicitly in terms of $(a_i)$, the elliptic operators $\mathcal{L}_1$ and $\mathcal{L}_2$ being of second order. The resulting criteria for the strong sign-preserving property of \eqref{i6} with homogeneous Dirichlet boundary conditions are, however, not straightforward to apply.

\medskip

In this paper, we revisit Schr\"oder's approach and show that, if the equation \eqref{i6} has an alternative formulation \eqref{i7} with $q\equiv 0$ and two second-order elliptic operators $\mathcal{L}_1$ and $\mathcal{L}_2$ satisfying the maximum principle, then \eqref{i6} with homogeneous Dirichlet boundary conditions enjoys the strong sign-preserving property. 
More precisely, we shall prove the following result:

\begin{theorem}\label{prbp1}
Consider two second-order elliptic operators 
\begin{equation*}
\mathcal{L}_i w := a_i(x)\ w'' + b_i(x)\ w' + c_i(x)\ w\,, \qquad x\in I:= (-1,1)\ , \qquad i=1,2\ ,
\end{equation*}
where $a_i$, $b_i$, and $c_i$, $i=1,2$, are bounded functions in $[-1,1]$ satisfying
\begin{equation}
\min\{a_1(x), a_2(x)\} \ge \eta>0 \;\;\text{ and }\;\; \max\{c_1(x), c_2(x)\} \le 0\,, \qquad x\in [-1,1]\ , \label{bp1}
\end{equation}
for some $\eta>0$. Assume that there is a pair of functions $(u,\gamma)$ satisfying
$$
u \in C^4(I)\cap C^2([-1,1])\ ,  \quad \gamma\in C^2(I)\cap C([-1,1])\ ,
$$ 
and 
\begin{eqnarray}
\mathcal{L}_1 u & = & \gamma \;\;\text{ in }\;\; I\,, \label{bp2} \\
\mathcal{L}_2 \gamma & \le & 0 \;\;\text{ in }\;\; I\,,\label{bp3} \\
u(\pm 1) & = & u'(\pm 1) = 0\,. \label{bp7}
\end{eqnarray}
Then 
\begin{equation}
\text{either }\quad u\equiv 0\quad \text{ or }\quad u < 0 \;\;\text{ in }\;\; I\,.
\label{bp8}
\end{equation}
\end{theorem}

An alternative way to formulate Theorem~\ref{prbp1} is the following: Given twice continuously differentiable functions $a_i, b_i, c_i$, $i=1,2$ on $[-1,1]$ satisfying \eqref{bp1}, the fourth-order differential operator
\begin{equation}\label{LL}
\mathcal{L} u= A_4(x) u'''' + A_3(x) u'''+ A_2(x) u'' + A_1(x) u' +A_0(x) u  
\end{equation}
with
\begin{align}
A_4& := a_1 a_2   \ ,& A_3&:= (2a_1'+b_1)a_2 +a_1 b_2		\ ,\nonumber\\
A_2&:= 	(a_1''+2b_1'+c_1)a_2 + (a_1'+b_1)b_2+a_1c_2		\ ,& & \label{A2}\\
A_1&:= 	(b_1''+2c_1') a_2 +(b_1'+c_1) b_2+b_1 c_2		\ , &
A_0&:= c_1'' a_2 + c_1'b_2+ c_1 c_2\nonumber
\end{align}
subject to homogeneous Dirichlet boundary conditions $u(\pm 1)=u'(\pm 1)=0$ enjoys the strong sign-preserving property.

\begin{remark}
On the one hand, it is unlikely that an arbitrary fourth-order elliptic operator $\mathcal{L}$ of the form \eqref{LL} has a decomposition \eqref{A2} as above. On the other hand, if there is such a decomposition, the choice of the functions $a_i, b_i, c_i$, $i=1,2$ in \eqref{A2} is clearly not unique and may play an important role. This feature is illustrated in Proposition~\ref{le.gr2} below.
\end{remark}

\begin{remark}
Theorem~\ref{prbp1} is stated in such generality that it applies to nonlinear operators as well, that is, the coefficients $a_i, b_i, c_i$ may depend on $u$ itself. In particular, we shall see in Subsection~\ref{Nonlinear} below that the Euler-Lagrange equation of the one-dimensional Willmore functional \cite{willmorebook} fits into the framework developed herein.
\end{remark}

The proof of Theorem~\ref{prbp1} is given in Section~\ref{sec2}. It is based on repeated applications of the strong maximum principle to $\mathcal{L}_1$ and $\mathcal{L}_2$ and requires a careful analysis of the possible behaviors of $\gamma$ (recall that the behavior of $\gamma$ on the boundary is unknown).
Though the proof is restricted to one space dimension, a similar result can be obtained for radially symmetric functions in higher dimensions:

\begin{theorem}\label{prrbp1}
Consider two second-order elliptic operators 
\begin{equation*}
\mathcal{L}_i w := \sum_{j,k=1}^d a_i^{jk}(x)\ \partial_j \partial_k w + \sum_{j=1}^d b_i^j(x)\ \partial_j w + c_i(x)\ w\,, \qquad x\in \mathbb{B}_1\ , \qquad i=1,2\ ,
\end{equation*}
where $\mathbb{B}_1$ denotes the unit ball of $\mathbb{R}^d$, $d\ge 1$, and $a_i^{jk} = a_i^{kj}$, $b_i^{j}$, and $c_i$, $i=1,2$, $1\le j,k\le d$, are bounded functions in $\mathbb{B}_1$ satisfying
\begin{equation}
\sum_{j,k=1}^d a_i^{jk}(x)\ \xi_j\ \xi_k \ge \eta |\xi|^2 \;\;\text{ and }\;\; c_i(x) \le 0\,, \qquad (x,\xi)\in \mathbb{B}_1\times\mathbb{R}^d\ , \label{rbp1}
\end{equation}
for some $\eta>0$ and $i=1,2$. Let $(u,\gamma)$ be a pair of functions satisfying
$$
u \in C^4(\mathbb{B}_1)\cap C^2(\bar{\mathbb{B}}_1)\ ,  \quad \gamma\in C^2(\mathbb{B}_1)\cap C(\bar{\mathbb{B}}_1)\ ,
$$ 
and 
\begin{eqnarray}
\mathcal{L}_1 u & = & \gamma \;\;\text{ in }\;\; \mathbb{B}_1\,, \label{rbp2} \\
\mathcal{L}_2 \gamma & \le & 0 \;\;\text{ in }\;\; \mathbb{B}_1\,, \label{rbp3} \\
u & = & \partial_n u = 0 \;\;\text{ on }\;\; \partial \mathbb{B}_1\,. \label{rbp7}
\end{eqnarray}
If $u$ and $\gamma$ are both radially symmetric, then 
\begin{equation}
\text{either }\quad u\equiv 0\quad \text{ or }\quad u < 0 \;\;\text{ in }\;\; \mathbb{B}_1\,.
\label{rbp8}
\end{equation}
\end{theorem}

The proof of Theorem~\ref{prrbp1} is closely related to that of Theorem~\ref{prbp1}, though it differs at some points. It is performed in Section~\ref{sec3}. The assumption that \eqref{rbp2}-\eqref{rbp7} has radially symmetric solutions clearly imposes certain constraints on the coefficients of the operators $\mathcal{L}_1$ and $\mathcal{L}_2$.\\

Theorem~\ref{prbp1} and Theorem~\ref{prrbp1} have several applications. For instance, Theorem~\ref{prbp1}
generalizes the strong sign-preserving property established in \cite{Ow97} and \cite[Proposition~1]{Gr02} for \eqref{i4} when $\lambda\le 0$, see Subsection~\ref{Sub43}.  Further applications of these theorems, e.g. to quasi-linear equations, are discussed in Section~\ref{sec4}. Let us single out one particular application to a semi-linear equation, which includes as a particular case a mathematical model for MEMS \cite{EGG10}:

\begin{theorem}\label{MEMS}
Let $J$ be an open interval in $\mathbb{R}$ containing $0$ and let $g\in C^2(J)$ be a non-negative and non-increasing function with $g(0)>0$. Let $B>0$ and $T>0$ be two real numbers. Then there exists $\lambda_*\in (0,\infty]$ such that, for any $\lambda\in (0,\lambda_*)$, the boundary value problem
$$
B\Delta^2 u-T\Delta u= -\lambda g(u)\ \text{ in } \mathbb{B}_1\ ,\quad u=\partial_n u=0\ \text{ on } \partial \mathbb{B}_1\ ,
$$
has a unique radially symmetric and non-positive classical maximal solution $u_\lambda \in C^4(\mathbb{B}_1)\cap C^2(\bar{\mathbb{B}}_1)$ such that $u_\lambda(\bar{\mathbb{B}}_1)\subset J$. If $\lambda_*<\infty$, then there is no radially symmetric classical solution for $\lambda>\lambda_*$. In addition, for each $x\in \mathbb{B}_1$, the function $\lambda\mapsto u_\lambda (x)$ is decreasing in $(0,\lambda_*)$. Furthermore, if $a:=\inf J>-\infty$ and $m:=\inf_{(a,0)} g >0$, then~$\lambda_*<\infty$.
\end{theorem}

The specific application to MEMS is obtained by taking $g(\xi)=1/(1+\xi)^2$ for $\xi\in (-1,1)$. We construct the maximal solution in Subsection~\ref{Sub46} by a monotonicity approach as in \cite[Chapter~11]{EGG10}, where a similar result is proven for $T=0$. An interesting intermediate step in the proof of Theorem~\ref{MEMS} is that there is a unique normalized radially symmetric positive eigenfunction corresponding to a positive eigenvalue of the operator $B\Delta^2 -T\Delta$ subject to homogeneous Dirichlet boundary conditions, see Subsection~\ref{Sub45}.

\section{Sign-Preserving Property in One Space Dimension}\label{sec2}

We prove Theorem~\ref{prbp1} and first state a simple consequence of the strong maximum principle for second-order operators.

\begin{lemma}\label{lebp2}
Consider $-1\le x_1<x_2\le 1$ and a function $v\in C^2((x_1,x_2))\cap C([x_1,x_2])$ such that $\mathcal{L}_1 v \le 0$ in $(x_1,x_2)$.
\begin{itemize}
\item[(a)] If $v(x_1)=v'(x_1)=0$, then either $v\equiv 0$ or $v(x_2)<0$.
\item[(b)] If $v(x_2)=v'(x_2)=0$, then either $v\equiv 0$ or $v(x_1)<0$.
\end{itemize}
\end{lemma}

\begin{proof}
The proof of the two assertions being similar, we only give that of~(a) and consider the case where $v(x_1)=v'(x_1)=0$. Assume for contradiction that there is $x_0\in (x_1,x_2]$ such that $v(x_0)> 0$. Recalling that $v(x_1)=0$, the function $v$ is obviously not constant in $(x_1,x_0)$. Since $ \mathcal{L}_1v\le 0$ in $(x_1,x_0)$, the strong maximum principle \cite[Theorem~3.5]{GT01} guarantees that $v$ cannot achieve a non-positive minimum in $(x_1,x_0)$ and thus $v(x)>0=v(x_1)$ for $x\in (x_1,x_0)$. We then infer from \cite[Lemma~3.4]{GT01} that $v'(x_1)>0$ and a contradiction. Consequently, 
\begin{equation}
v(x)\le 0 \;\text{  for all }\; x\in (x_1,x_2]\ . \label{bp4}
\end{equation}
 In particular, $v(x_2)\le 0$. Then either $v(x_2)<0$ and the second alternative in Lemma~\ref{lebp2}~(a) is true. Or $v(x_2)=0=v(x_1)$ and the minimum principle entails that $v\ge 0$ in $(x_1,x_2)$. Combining this fact with \eqref{bp4} gives $v\equiv 0$ and completes the proof.
\end{proof}

\medskip

\begin{proof}[Proof of Theorem~\ref{prbp1}]
If $u\equiv 0$, then there is nothing to prove and we thus may assume $u\not\equiv 0$. It first follows from \eqref{bp3} and the minimum principle \cite[Corollary~3.2]{GT01} that 
\begin{equation}
\min_{[-1,1]} \gamma \ge \min{\left\{ \gamma(-1)\wedge 0 , \gamma(1)\wedge 0 \right\}}\,. \label{bp9}
\end{equation}

Assume first for contradiction that $\min_{[-1,1]} \gamma\ge 0$. Then $\gamma\ge 0$ in $I$, which gives, together with \eqref{bp2}, \eqref{bp7}, the property $u\not\equiv 0$, and the strong maximum principle \cite[Theorem~3.5]{GT01} that $u$ cannot achieve a non-negative maximum in $(-1,1)$ and thus $u(x)<0=u(1)$ for all $x\in I$. It then follows from \cite[Lemma~3.4]{GT01} that $u'(1)>0$ which contradicts \eqref{bp7}. Therefore, $\min_{[-1,1]} \gamma< 0$ and it follows from \eqref{bp9} that $\gamma(-1)\wedge \gamma(1)<0$. We may assume without loss of generality that
\begin{equation}
\gamma(-1) = \gamma(-1)\wedge \gamma(1) = \min_{[-1,1]} \gamma < 0. \label{bp10}
\end{equation}
In addition, since $u(-1)=u'(-1)=0$, $u(1)=0$, and $u\not\equiv 0$, Lemma~\ref{lebp2} implies that 
\begin{equation}
\max_{[-1,1]} \gamma > 0\,. \label{bp12}
\end{equation}
Owing to \eqref{bp10} and \eqref{bp12},
\begin{equation}
y_0 := \inf{\left\{ x \in I\ :\ \gamma(x)>0 \right\}} \in I \label{bp14}
\end{equation}
and
\begin{equation}
\gamma\le 0 \;\;\text{ in }\;\; (-1,y_0) \;\;\text{ and }\;\; \gamma(y_0)=0\,. \label{bp15}
\end{equation}
On the one hand, for $y\in (-1,y_0]$, we obtain from \eqref{bp2} and \eqref{bp15} that $\mathcal{L}_1 u \le 0$ in $(-1,y)$, and \eqref{bp7} and Lemma~\ref{lebp2} imply that either $u(y)<0$ or $u\equiv 0$ in $(-1,y)$. But the latter cannot occur as it would imply $\gamma \equiv 0$ in $(-1,y)$ by \eqref{bp2} contradicting \eqref{bp10}. We have thus shown that
\begin{equation}
u<0 \;\;\text{ in }\;\; (-1,y_0]\,. \label{bp17}
\end{equation}
On the other hand, consider $y\in (y_0,1)$ such that $\gamma(y)>0$ (the existence of such points is guaranteed by the definition of $y_0$). According to \eqref{bp3} and the strong maximum principle \cite[Theorem~3.5]{GT01}, the function $\gamma$ cannot achieve a non-positive minimum in $(y_0,y)$ unless it is constant. Since $\gamma(y)>\gamma(y_0)=0$, we conclude that $\gamma>0$ in $(y_0,y)$. We have thus shown that
$$
y_1 := \sup{\left\{ y\in (y_0,1)\ :\ \gamma>0 \;\text{ in }\; (y_0,y) \right\}} > y_0\ .
$$
Assume now for contradiction that $y_1=1$. Then $\gamma> 0$ in $(y_0,1)$ and, since $u$ is not constant by \eqref{bp7} and \eqref{bp17}, it follows from \eqref{bp2}, \eqref{bp7}, \eqref{bp17}, and the strong maximum principle that $u$ cannot achieve a non-negative maximum in $(y_0,1)$, whence $u<0=u(1)$ in $(y_0,1)$. We then infer from \cite[Lemma~3.4]{GT01} that $u'(1)>0$ and a contradiction. Consequently, 
\begin{equation}
y_1\in (y_0,1)\ , \qquad \gamma>0 \;\text{ in }\; (y_0,y_1)\ , \qquad \gamma(y_1)=0\ . \label{bp18}
\end{equation}
In particular, this together with \eqref{bp3} and \cite[Lemma~3.4]{GT01} imply 
\begin{equation}\label{y1}
\gamma'(y_1)<0\ .
\end{equation}
Hence,
 there is a sequence $(q_n)_{n\ge 1}$ such that $q_n\in (y_1,y_1+1/n)\cap I$ and $\gamma(q_n)<0$ for all $n\ge 1$. Fix $n\ge 1$ and assume for contradiction that there is $p_n\in (q_n,1)$ such that $\gamma(p_n)>0$. Owing to the continuity of $\gamma$, there is $z_n\in (q_n,p_n)$ such that $\gamma(z_n)=0$. It then follows from \eqref{bp3} that $\mathcal{L}_2 \gamma \le 0$ in $(y_1,z_n)$ with $\gamma(y_1)=\gamma(z_n)=0>\gamma(q_n)$ which contradicts the strong maximum principle. Consequently, $\gamma\le 0$ in $(q_n,1)$ for all $n\ge 1$ and we have established that 
\begin{equation}
\gamma\le 0 \;\;\text{ in }\;\; (y_1,1)\ . \label{bp19}
\end{equation} 
Now, fix $y\in [y_1,1)$. Owing to \eqref{bp2}, \eqref{bp7}, and \eqref{bp19}, we are in a position to apply Lemma~\ref{lebp2} in $(y,1)$ and conclude that either $u(y)<0$ or $u\equiv 0$ in $(y,1)$. If the latter occurs, it follows from \eqref{bp2} that $\gamma\equiv 0$ in $(y,1)$ as well. Therefore, $\mathcal{L}_2 \gamma \le 0$ in $(y_1,y)$ with $\gamma(y_1)=\gamma(y)=0$ which entails $\gamma\ge 0$ in $(y_1,y)$ by the minimum principle and contradicts \eqref{y1}. Consequently,
\begin{equation}
u<0 \;\;\text{ in }\;\; [y_1,1)\,. \label{bp21}
\end{equation}
Finally, recalling that $\gamma>0$ in $(y_0,y_1)$ by \eqref{bp18}, we infer from \eqref{bp2} and the strong maximum principle that either $u$ is constant in $(y_0,y_1)$ and then $u\equiv u(y_0)=u(y_1)<0$ by \eqref{bp17} and \eqref{bp21} or $u$ cannot achieve a non-negative maximum in $(y_0,y_1)$. In both cases, we conclude $u<0$ in $(y_0,y_1)$, which, together with \eqref{bp17} and \eqref{bp21} gives \eqref{bp8} and completes the proof.
\end{proof}

\begin{corollary}\label{c1}
Under the assumptions of Theorem~\ref{prbp1} and if $u\not\equiv 0$, there are $-1<y_0<y_1<1$ such that $\gamma$ satisfies 
$$
\gamma>0\text{ in } (y_0,y_1)\ ,\quad \gamma\le 0\text{ in } (-1,y_0]\cap [y_1,1) \ ,
$$
and $\gamma(\pm 1)<0$. Furthermore, $u''(\pm 1)<0$.
\end{corollary}

\begin{proof}
It follows from \eqref{bp15}, \eqref{bp18}, and \eqref{bp19} that $\gamma$ satisfies all the above properties except for the property $\gamma(1)<0$. However, we know from \eqref{bp19} that $\gamma(1)\le 0$ and assuming for contradiction that $\gamma(1)= 0$, the strong maximum principle and \eqref{bp3} entail $\gamma\equiv 0$ in $(y_1,1)$ which is not possible according to \eqref{y1}. Finally, it follows from \eqref{bp2} and \eqref{bp7} that $a_1(\pm 1) u''(\pm 1)=\gamma(\pm 1)$, which gives $u''(\pm 1)<0$ thanks to $\gamma(\pm 1)<0$ and \eqref{bp1}.
\end{proof}


\section{Sign-Preserving Property in Radial Symmetry}\label{sec3}

This section is dedicated to the proof of Theorem~\ref{prrbp1}.
For further use, we set $\mathbb{B}_r:= \{ x\in\mathbb{R}^d\ :\ |x|<r\}$, $\bar{\mathbb{B}}_r := \{ x\in\mathbb{R}^d\ :\ |x|\le r\}$, and $\mathbb{S}_r := \{ x\in\mathbb{R}^d\ :\ |x|=r\}$ for $r\in [0,1)$. Let us stress that, throughout this section, we only deal with radially symmetric functions and often use the property that such functions are constant on $\mathbb{S}_r$ for any $r$ without further notice.

We first establish a variant of Lemma~\ref{lebp2}.

\begin{lemma}\label{lerbp2}
Consider $\varrho\in [0,1)$ and a radially symmetric function $v\in C^2(\mathbb{B}_1\setminus \bar{\mathbb{B}}_\varrho)\cap C(\bar{\mathbb{B}}_1\setminus \mathbb{B}_\varrho)$ such that 
\begin{equation}
\mathcal{L}_1 v \le 0 \;\text{ in }\; \mathbb{B}_1\setminus \bar{\mathbb{B}}_\varrho \;\text{ and }\; v=\partial_n v=0 \;\text{ on }\; \partial \mathbb{B}_1\ . \label{rbp10}
\end{equation}
Then either $v\equiv 0$ in $\mathbb{B}_1\setminus \bar{\mathbb{B}}_\varrho$ or $v(x)<0$ for $x\in\partial \mathbb{B}_\varrho$.
\end{lemma}

\begin{proof}
Assume for contradiction that there is $r\in [\varrho,1)$ such that $v(x)> 0$ for $x\in \mathbb{S}_r$. Recalling that $v=0$ on $\mathbb{S}_1$, the function $v$ is obviously not constant in $\mathbb{B}_1\setminus\bar{\mathbb{B}}_\varrho$. Since $\mathcal{L}_1 v\le 0$ in $(x_1,x_0)$,  the strong maximum principle \cite[Theorem~3.5]{GT01} guarantees that $v$ cannot achieve a non-positive minimum in $\mathbb{B}_1\setminus\bar{\mathbb{B}}_\varrho$ and thus $v(x)>0$ for $x\in \mathbb{B}_1\setminus\bar{\mathbb{B}}_\varrho$. Consequently, if $x_0\in\mathbb{S}_1$, the previous property and \eqref{rbp10} imply that $v(x_0)<v(x)$ for $x\in \mathbb{B}_1\setminus\bar{\mathbb{B}}_\varrho$. We are then in a position to apply \cite[Lemma~3.4]{GT01} and conclude $\partial_n v(x_0)<0$, which contradicts \eqref{rbp10}. Therefore, 
\begin{equation}
v(x)\le 0 \;\text{  for all }\; x\in \mathbb{B}_1\setminus\bar{\mathbb{B}}_\varrho\ . \label{rbp4}
\end{equation}
Now, either $v<0$ on $\mathbb{S}_\varrho$ and the second alternative in Lemma~\ref{lerbp2}~(a) is true. Or $v=0$  on $\mathbb{S}_\varrho$ and the minimum principle entails that $v\ge 0$ in $\mathbb{B}_1\setminus\bar{\mathbb{B}}_\varrho$. Combining this fact with \eqref{rbp4} gives $v\equiv 0$ and completes the proof.
\end{proof}

\medskip

\begin{proof}[Proof of Theorem~\ref{prrbp1}]
We may again assume that $u\not\equiv 0$. It follows from \eqref{rbp3} and the minimum principle \cite[Corollary~3.2]{GT01} that 
\begin{equation}
\min_{\bar{\mathbb{B}}_1} \gamma \ge \min_{\mathbb{S}_1}{\left\{ \gamma\wedge 0 \right\}}\,. \label{rbp9}
\end{equation}
Assume first for contradiction that $\min_{\bar{\mathbb{B}}_1} \gamma\ge 0$. Then $\gamma\ge 0$ in $\mathbb{B}_1$ and, since $u\not\equiv 0$ in $\mathbb{B}_1$, we deduce from \eqref{rbp2}, \eqref{rbp7}, and the strong maximum principle \cite[Theorem~3.5]{GT01} that $u$ cannot achieve a non-negative maximum in $\mathbb{B}_1$ and thus $u<0$ in $\mathbb{B}_1$. Therefore, given $x_0\in \mathbb{S}_1$, the previous property and \eqref{rbp7} guarantee that $u(x)<u(x_0)=0$ for $x\in \mathbb{B}_1$. Using \cite[Lemma~3.4]{GT01}, we conclude that $\partial_n u(x_0)>0$ and a contradiction with \eqref{rbp7}. We have thus proved that 
\begin{equation}\label{zzz}
\min_{\bar{\mathbb{B}}_1} \gamma< 0\ .
\end{equation}
Then $\gamma(x)<0$ for $x\in\mathbb{S}_1$ by \eqref{rbp9} and the radial symmetry of $\gamma$, from which we deduce that
\begin{equation*}
r_0 := \inf{\left\{ r \in [0,1)\ :\ \gamma(x)<0 \;\text{ for }\; x\in\mathbb{S}_r \right\}} <1\ .
\end{equation*}
Assume for contradiction that $r_0=0$. Then $\gamma\le 0$ in $\mathbb{B}_1$ and \eqref{rbp2}, \eqref{rbp7}, and the minimum principle ensure that $u\ge 0$ in $\mathbb{B}_1$ while Lemma~\ref{lerbp2} implies that $u$ is non-positive in $\mathbb{B}_1$. Consequently, $u\equiv 0$ in $\mathbb{B}_1$ and a contradiction. We have thus shown that
\begin{equation}
r_0>0\ , \quad \gamma = 0 \;\text{ on }\; \mathbb{S}_{r_0}\ , \;\text{ and }\; \gamma<0 \;\text{ in }\; \mathbb{B}_1\setminus\bar{\mathbb{B}}_{r_0}\ . \label{rbp11}
\end{equation}

Now, on the one hand, let $r\in [r_0,1)$. By \eqref{rbp2}, \eqref{rbp7}, \eqref{rbp11}, and Lemma~\ref{lerbp2}, we have either $u<0$ on $\mathbb{S}_r$ or $u\equiv 0$ in $\mathbb{B}_1\setminus\bar{\mathbb{B}}_r$. However, the latter and \eqref{rbp2} would imply $\gamma\equiv 0$ in $\mathbb{B}_1\setminus\bar{\mathbb{B}}_r$ contradicting \eqref{rbp11}. Consequently,
\begin{equation}
u<0 \;\text{ in }\; \mathbb{B}_1\setminus \mathbb{B}_{r_0}\ . \label{rbp12}
\end{equation}
On the other hand, take $r\in (0,r_0)$ such that $\gamma>0$ on $\mathbb{S}_r$ (such an $r$ exists according to the definition and positivity of $r_0$). Since $\gamma$ is obviously non-constant in $\mathbb{B}_{r_0}\setminus\bar{\mathbb{B}}_r$, it follows from the strong maximum principle that $\gamma$ cannot achieve a non-positive minimum in $\mathbb{B}_{r_0}\setminus\bar{\mathbb{B}}_r$ and thus $\gamma>0$ in $\mathbb{B}_{r_0}\setminus\bar{\mathbb{B}}_r$. Introducing
$$
r_1 := \inf{\left\{ r \in (0,r_0)\ :\ \gamma>0 \;\text{ in }\; \mathbb{B}_{r_0}\setminus\bar{\mathbb{B}}_r \right\}} >0\ ,
$$
we have just proved that $r_1<r_0$. Assume now for contradiction that $r_1>0$. Then $\gamma=0$ on $\mathbb{S}_{r_1}$, which, together with \eqref{rbp3} and the minimum principle, entails $\gamma\ge 0$ in $\mathbb{B}_{r_1}$. However, since $\mathcal{L}\gamma \le 0$ in $\mathbb{B}_{r_0}\setminus \mathbb{B}_{r_1}$, $\gamma>0$ in $\mathbb{B}_{r_0}\setminus \mathbb{B}_{r_1}$, and $\gamma=0$ on $\mathbb{S}_{r_1}$, it follows from \cite[Lemma~3.4]{GT01} that $\partial_n\gamma<0$ on $\mathbb{S}_{r_1}$. Hence, we deduce $\gamma<0$ on $\mathbb{S}_{r}$ for $r$ close to $r_1$ with $r<r_1$ and a contradiction. Therefore, $r_1=0$ and $\gamma>0$ in $\mathbb{B}_{r_0}\setminus\{0\}$. It then follows from \eqref{rbp2}, \eqref{rbp12}, and the comparison principle that
$$
\max_{\bar{\mathbb{B}}_{r_0}}{u} \le \max_{\mathbb{S}_{r_0}}{\left\{ u \vee 0 \right\}} = 0\ .
$$
Now, either $u$ is constant in $\mathbb{B}_{r_0}$ and we deduce from \eqref{rbp12} that $u<0$ in $\mathbb{B}_{r_0}$. Or $u$ is not constant in $\mathbb{B}_{r_0}$ and we infer from the strong maximum principle that $u$ cannot achieve a non-negative maximum in $\mathbb{B}_{r_0}$, that is, $u<0$ in $\mathbb{B}_{r_0}$. Recalling \eqref{rbp12}, we have proved \eqref{rbp8}.
\end{proof}

As in Corollary~\ref{c1} for the one-dimensional case, we deduce from the proof of Theorem~\ref{prrbp1} additional information on $u$ on the boundary. Actually, for the biharmonic operator with homogeneous boundary conditions the following result is proved in \cite{GS96b} by means of the integral representation and the positivity properties of the Green function.

\begin{corollary}\label{c22}
Under the assumptions of Theorem~\ref{prrbp1} and if $u\not\equiv 0$, one has $U''(1)<0$, where $U(\vert x\vert )=u(x)$ for $x\in \bar{\mathbb{B}}_1$.
\end{corollary}

\begin{proof}
Let $x\in \mathbb{S}_1$ be fixed.
On the one hand, we deduce from \eqref{rbp9} and \eqref{zzz} that $\gamma(x)<0$. On the other hand,
since  $U$ and $U'$ both vanish at $r= 1$, we infer from  \eqref{rbp2} that
$$
0>\gamma(x)=\sum_{j,k=1}^d a_1^{jk}(x)\, x_j\, x_k\, U''(1) 
$$
and the claim readily follows from \eqref{rbp1}.
\end{proof}

\section{Applications} \label{sec4}

In this section we collect a few rather immediate consequences of the previous results. We first give in Subsection~\ref{sub41} a straightforward extension of Theorem~\ref{prrbp1} to an annulus. We then show in Subsection~\ref{Nonlinear} that Theorem~\ref{prbp1} can also be applied in a quasi-linear framework. Moreover, we use Theorem~\ref{prbp1} and Theorem~\ref{prrbp1} to derive the sign-preserving property for \eqref{i5} for a certain range of parameter values (Subsection~\ref{Sub43}), to characterize the polar cone of positive functions related to Moreau's decomposition (Subsection~\ref{sub44}), and to establish the existence of positive eigenfunctions (Subsection~\ref{Sub45}).

\subsection{Equations in an Annulus with Radial Symmetry}\label{sub41}

We prove a result similar to Theorem~\ref{prrbp1} in an annulus $\mathbb{B}_1\setminus \bar{\mathbb{B}}_\rho$ with $\rho\in (0,1)$.

\begin{corollary}\label{corrbp1}
Let $\rho\in (0,1)$ and
consider two second-order elliptic operators 
\begin{equation*}
\mathcal{L}_i w := \sum_{j,k=1}^d a_i^{jk}(x)\ \partial_j \partial_k w + \sum_{j=1}^d b_i^j(x)\ \partial_j w + c_i(x)\ w\,, \qquad x\in \mathbb{B}_1\setminus \bar{\mathbb{B}}_\rho\ , \qquad i=1,2\ ,
\end{equation*}
where  the coefficients $a_i^{jk}=a_i^{kj}$, $b_i^{j}$, and $c_i$, $i=1,2$, $1\le j,k\le d$ are bounded functions. Further assume that the functions $A_i$ and $B_i$, defined for $x\in \mathbb{B}_1\setminus \bar{\mathbb{B}}_\rho$ by
\begin{align*}
 A_i(x)  &:= \sum_{j,k=1}^{d} a_i^{jk}(x)\frac{x_j x_k}{\vert x\vert^2}  \ ,\\
 B_i(x)  &:= \frac{1}{\vert x\vert} \left(\sum_{j=1}^d a_i^{jj}(x)- \sum_{j,k=1}^{d} a_i^{jk}(x)\frac{x_j x_k}{\vert x\vert^2}\right)+\sum_{j=1}^d b_i^j(x)\frac{x_j}{\vert x\vert}\ ,
\end{align*}
and $c_i$ are rotationally invariant in the annulus $\mathbb{B}_1\setminus \bar{\mathbb{B}}_\rho$ and satisfy
$$
A_i(x)\ge \eta >0\ \text{ and } \ c_i(x)\le 0\ ,\quad x\in \mathbb{B}_1\setminus \bar{\mathbb{B}}_\rho\ ,\quad i=1,2\ .
$$ 
Let $(u,\gamma)$ be a pair of functions satisfying
$$
u \in C^4(\mathbb{B}_1\setminus \bar{\mathbb{B}}_\rho)\cap C^2(\bar{\mathbb{B}}_1\setminus \mathbb{B}_\rho)\ , \quad \gamma\in C^2(\mathbb{B}_1\setminus \bar{\mathbb{B}}_\rho)\cap C(\bar{\mathbb{B}}_1\setminus \mathbb{B}_\rho)\ ,
$$ 
and 
\begin{eqnarray}
\mathcal{L}_1 u & = & \gamma \;\;\text{ in }\;\; \mathbb{B}_1\setminus \bar{\mathbb{B}}_\rho\,, \label{rbp22} \\
\mathcal{L}_2 \gamma & \le & 0 \;\;\text{ in }\;\; \mathbb{B}_1\setminus \bar{\mathbb{B}}_\rho\,, \label{rbp33} \\
u & = & \partial_n u = 0 \;\;\text{ on }\;\; \partial (\mathbb{B}_1\setminus \bar{\mathbb{B}}_\rho)\,. \label{rbp77}
\end{eqnarray}
If $u$ and $\gamma$ are both radially symmetric, then 
\begin{equation}
\text{either }\quad u\equiv 0 \quad\text{ or }\quad u < 0  \quad\;\;\text{ in }\;\; \mathbb{B}_1\setminus \bar{\mathbb{B}}_\rho\,.
\label{rbp88}
\end{equation}
\end{corollary}

\begin{proof}
Owing to the radial symmetry of $u$ and $\gamma$, we may write $u(x)=U(\vert x\vert)$ and $\gamma(x)=\Gamma(\vert x\vert )$ and deduce from \eqref{rbp22}-\eqref{rbp77} that
\begin{align*}
 A_1(r) U'' + B_1(r) U' + C_1(r) U &= \Gamma \ , &&  r\in (\rho,1)\ ,\\
 A_2(r) \Gamma'' + B_2(r) \Gamma' + C_2(r) \Gamma &= f(r) \ ,&&  r\in (\rho,1)\ ,\\
 U(\rho)=U(1)=U'(\rho)=U'(1)&= 0\ ,  &&
\end{align*}
where $f:= \mathcal{L}_2\gamma$ and $C_i(\vert x\vert ):= c_i(x)$ for $x\in \mathbb{B}_1\setminus \bar{\mathbb{B}}_\rho$.
On $\mathbb{B}_1\setminus \bar{\mathbb{B}}_\rho$, the functions $A_i$, $B_i$, and $C_i$ are bounded for $i=1,2$, and, due to $f\le 0$ with $f\not\equiv 0$, we are in a position to apply Theorem~\ref{prbp1} to conclude that either $U\equiv 0$ or $U<0$ in $(\rho,1)$.
\end{proof}

\begin{remark}
We shall point out that radial symmetry is important to derive the strong sign-preserving property stated in Corollary~\ref{corrbp1}. Indeed, this property is no longer true for arbitrary functions in an annulus $\mathbb{B}_1\setminus\bar{\mathbb{B}}_\rho$ if $\rho>0$ is sufficiently small \cite{NS77}.
\end{remark}

\subsection{An Application to a Nonlinear Equation}\label{Nonlinear}

Let $B>0$, $T\ge 0$, $\alpha>0$, and consider the quasi-linear equation
\begin{equation}\label{willmore}
B\ \left( \frac{u''}{\left( 1 + (u')^2 \right)^\alpha}\right)'' + \alpha B\ \left( \frac{u' (u'')^2}{\left( 1 + (u')^2 \right)^{\alpha+1}}\right)' - T\ \left( \frac{u'}{\sqrt{1+(u')^2}}\right)' = f \;\;\text{ in }\;\; I\,,
\end{equation}
subject to homogeneous Dirichlet boundary conditions $u(\pm 1)= u'(\pm 1)=0$. Taking $B=1$, $T=0$, and $\alpha=5/2$ one obtains the Euler-Lagrange equation for the one-dimensional Willmore functional for a graph $u$, this functional actually dating back to D.~Bernoulli \cite{TR83}. Also, for $B>0$, $T>0$, and $\alpha=5/2$, equation \eqref{willmore} describes deformations of a membrane clamped at its boundary when the assumption of small deformations is dropped (recall that this assumption allows one to replace \eqref{willmore} by its linearization \eqref{i4}). Now, if $u$ is a classical solution to \eqref{willmore} with right-hand side $f\le 0$ and $f\not\equiv 0$, then $u<0$ in $I$.
 Indeed, this follows from Theorem~\ref{prbp1} by defining
$$
\gamma := \frac{u''}{\left( 1 + (u')^2 \right)^{\alpha/2}}\,,
$$
a transformation already used by Euler \cite[p.248]{Euler}, and setting $a_1=\left( 1 + (u')^2 \right)^{-\alpha/2}$, $b_1=c_1=0$ and
$$
a_2 = B\ \left( 1 + (u')^2 \right)^{-\alpha/2}\,, \quad b_2=a_2'\,, \quad c_2=-T\ \left( 1 + (u')^2 \right)^{(\alpha-3)/2}\,.
$$
Equation \eqref{willmore} with $f\equiv 0$ and  possibly inhomogeneous Dirichlet boundary conditions has been studied
in \cite{DG07}, where also a more detailed account of related research can be found.

\subsection{Equations with Anti-Diffusive Lower Order Terms}\label{Sub43}

As mentioned in the introduction, it is straightforward to deduce from Theorem~\ref{prbp1} that the boundary value problem \eqref{i5} is strongly sign-preserving when $\lambda\le 0$ and $a\in\RR$. This follows simply by choosing $(a_1,b_1,c_1)=(1,0,0)$ and $(a_2,b_2,c_2)=(1,a,\lambda)$.  We show in the next result that the range of $\lambda$, for which the strong sign-preserving property is true, can be extended to some positive values of $\lambda$. To this end, an alternative choice of $\mathcal{L}_1$ and $\mathcal{L}_2$ is required.

\begin{proposition}\label{le.gr2}
Consider $\lambda\in (0,(a^2+\pi^2)/4)$ with $a\in\RR$. Then the boundary value problem \eqref{i5}
\begin{equation*}
u'''' + a u''' + \lambda u''= f \;\text{ in }\; I\ , \qquad u(\pm 1) = u'(\pm 1) = 0 \ , \label{i10}
\end{equation*}
enjoys the strong sign-preserving property. 
\end{proposition}

\begin{proof}
If $p\in C^2([-1,1])$ is any positive function, then the function $\gamma := u''/p$ is well-defined in $I$ and it readily follows from \eqref{i5} that $\gamma$ solves
$$
p \gamma'' + \left( 2 p' + a p \right) \gamma´ + \left( p'' + a p'+ \lambda p \right) \gamma = f \;\text{ in }\; I\ .
$$
Define now
$$
p(x) := \left\{
\begin{array}{lcl}
e^{-(a+\sqrt{a^2-4\lambda}) x/2} & \text{ if } & 0 < \lambda < a^2/4\ , \\
(2+x) e^{-ax/2} & \text{ if } & \lambda = a^2/4\ , \\
\cos{(\sqrt{4\lambda-a^2}\ x/2)} e^{-ax/2} & \text{ if } & a^2/4 < \lambda\ ,
\end{array}
\right. 
$$ for $x\in I$ and observe that for $a^2/4<\lambda<(a^2+\pi^2)/4$ we have $p(x)\ge \cos{(\sqrt{4\lambda-a^2}/2)} e^{-a/2}>0$ for $x\in [-1,1]$, the positivity of $p$ being obvious in the other two cases. 

Consequently,  $p>0$ in $[-1,1]$ and, since $p'' + a p' + \lambda p=0$ in $I$, we are in a position to apply Theorem~\ref{prbp1} with $(a_1,b_1,c_1) = (1/p,0,0)$ and $(a_2,b_2,c_2) = (p,2p'+ap,0)$ and conclude that $u<0$ in $I$ if $f\le 0$ in $I$ with $f\not\equiv 0$. 
\end{proof}

\begin{remark}\label{R1}
When $a=0$, we partly recover \cite[Corollary~1]{Gr02}, where the strong sign-preserving property was shown for \eqref{i5} for the wider range $\lambda\in (0,\pi^2)$. However, it may be that this result can be fully recovered by an alternative choice of the operators $\mathcal{L}_1$ and $\mathcal{L}_2$ than the one used above. The result for $a\ne 0$ seems to be new.
\end{remark}

\subsection{An Application to Moreau's Decomposition}\label{sub44}

Let $B>0$ and $T\ge 0$ and define the scalar product $\langle \cdot ,\cdot\rangle$ on $H_D^2(I):=\{u\in H^2(I)\,;\, u(\pm 1)=u'(\pm 1)=0\}$ by
$$
\langle u ,v\rangle :=\int_{-1}^1 \big[ B u''(x) v''(x)+ T u'(x) v'(x)\big]\, \mathrm{d} x\ ,\quad u, v\in H_D^2(I)\ .
$$
Let $\mathcal{K}:=\{u\in H_D^2(I)\,;\, u\ge 0\}$ be the positive cone of $H_D^2(I)$.
According to Moreau's decomposition~\cite{Mo62}, we can write any $u\in H_D^2(I)$ in a unique way as a sum $u=v+w$ with  $\langle v,w\rangle =0$, where $v$ belongs to $\mathcal{K}$ and  $w$ belongs to its polar cone $$\mathcal{K}^\circ:=\{w\in H_D^2(I)\,;\, \langle z ,w\rangle \le 0\ \text{ for all } z\in \mathcal{K}\}\ , $$ which can be further characterized:

\begin{proposition}\label{uu}
We have $\mathcal{K}^\circ \subset -\mathcal{K}$.
\end{proposition}

As a consequence of Proposition~\ref{uu}, any function in $H_D^2(I)$ can be written as the sum of a non-negative and a non-positive function, which are orthogonal with respect to $\langle\cdot,\cdot\rangle$.

\begin{proof}
Let $w\in \mathcal{K}^\circ$ and let $f\in \mathcal{K}\cap C_0^\infty(I)$. By Theorem~\ref{prbp1}, the classical solution $u$ of the equation $$Bu''''-Tu''=f\  \text{  in } I\ ,\quad u(\pm 1)=u'(\pm 1)=0\ ,$$ 
belongs to $\mathcal{K}$, whence $\langle u ,w\rangle\le 0$. Integration by parts gives $$\int_{-1}^1 f(x) w(x)\,\mathrm{d}x \le 0\ ,$$
and since $f\in \mathcal{K}\cap C_0^\infty(I)$ was arbitrary, this yields the statement.
\end{proof}

Similar results can be found in \cite[Lemma~11.1.4]{EGG10} and \cite[Proposition~3.6]{GGS10} when $T=0$.

\subsection{Positive Eigenfunctions}\label{Sub45}

As for second-order elliptic operators we can combine the strong sign-preserving property with the celebrated Kre\u\i n-Rutman theorem to obtain information on the principal eigenvalue of some fourth-order elliptic operators. We first consider the one-dimensional case.

\begin{theorem}\label{wildschwein}
Let $a_i$, $b_i$, and $c_i$, $i=1,2$ be four times continuously differentiable functions on $[-1,1]$ satisfying \eqref{bp1}. Let $\mathcal{L}$ be the fourth-order elliptic operator defined in \eqref{LL} with coefficients given in \eqref{A2}. Then the eigenvalue problem 
$$
\mathcal{L}\phi =\mu \phi\ \text{ in } I\ ,\quad \phi(\pm 1)= \phi'(\pm 1)=0\ ,
$$
has an eigenvalue $\mu_1>0$ with a corresponding eigenfunction $\phi_1>0$ in $I$, and $\mu_1$ is the only eigenvalue with a positive eigenfunction. Moreover, any eigenfunction corresponding to the eigenvalue $\mu_1$ is a scalar multiple of $\phi_1$.
\end{theorem}

\begin{proof}
It readily follows from Theorem~\ref{prbp1} that the equation $\mathcal{L} u=0$ in $I$ subject to the boundary condition  $u(\pm 1)= u'(\pm 1)=0$ has only the trivial solution. Therefore, since the coefficients of $\mathcal{L}$ are $C^2$-smooth functions, we may apply \cite[Theorem~2.19]{GGS10} to obtain that, for any fixed $\alpha\in (0,1)$ and any $f\in C^{1+\alpha}([-1,1])$, the equation 
$$
\mathcal{L} u=f\ \text{ in } I\ ,\quad u(\pm 1)= u'(\pm 1)=0\ ,
$$  
has a unique solution $u\in C^{4+\alpha}(I)\cap C^{2+\alpha}([-1,1])$, which we denote by $Kf$. Moreover, $K$ belongs to $\mathcal{L}(C^{1+\alpha}([-1,1]), C^{4+\alpha}([-1,1]))$. Let $C_D^{2+\alpha}([-1,1])$ be the closed subspace consisting of all functions $w\in C^{2+\alpha}([-1,1])$ such that $w(\pm 1)= w'(\pm 1)=0$. The previous property and the Arzel\`a-Ascoli theorem entail that $K$ is a compact endomorphism of $C_D^{2+\alpha}([-1,1])$. In addition, given a non-negative $f\in C_D^{2+\alpha}([-1,1])$ with $f\not\equiv 0$, Theorem~\ref{prbp1} and Corollary~\ref{c1} guarantee that $Kf>0$ in $I$ with $(Kf)''(\pm 1)>0$. Thanks to these two properties, a standard argument shows that $Kf$ belongs to the interior of the positive cone of $C_D^{2+\alpha}([-1,1])$. Consequently, $K$ is strongly positive, and we may apply the Kre\u\i n-Rutman theorem (see e.g. \cite[Theorem~3.2]{A76}) to conclude the statement.
\end{proof}

An analogous result holds in the a higher-dimensional radially symmetric case. For the sake of simplicity, we only state and prove it for the particular operator $B\Delta^2 -T\Delta $ with homogeneous Dirichlet boundary conditions.

\begin{theorem}\label{asterix}
Let $B>0$ and $T\ge 0$.  Then the eigenvalue problem 
\begin{equation}\label{obelix}
B\Delta^2\phi -T\Delta\phi =\mu \phi\ \text{ in } \mathbb{B}_1\ ,\quad \phi= \partial_n\phi=0\ \text{ on } \partial \mathbb{B}_1\ ,
\end{equation}
has an eigenvalue $\mu_1>0$ with a corresponding radially symmetric eigenfunction $\phi_1>0$ in $\mathbb{B}_1$, and $\mu_1$ is the only eigenvalue with a positive radially symmetric eigenfunction. Moreover, any radially symmetric eigenfunction corresponding to the eigenvalue $\mu_1$ is a scalar multiple of $\phi_1$.
\end{theorem}

\begin{proof}
We first check that if $u$ solves \eqref{obelix} with $\mu=0$, then $u\equiv 0$ in $\mathbb{B}_1$. Indeed, multiplying the equation by $u$ and integrating by parts give $B\Delta u \equiv 0$ in $\mathbb{B}_1$ with $u=0$ on $\partial \mathbb{B}_1$, whence $u\equiv 0$ in $\mathbb{B}_1$. Consequently, as in the proof of Theorem~\ref{wildschwein} we may apply \cite[Theorem~2.19]{GGS10} to obtain that, for any fixed $\alpha\in (0,1)$ and any $f\in C^{1+\alpha}(\bar{\mathbb{B}}_1)$, the equation 
\begin{equation}\label{hinkelstein}
B\Delta^2 u -T\Delta u=f\ \text{ in } \mathbb{B}_1\ ,\quad u= \partial_n u=0\ \text{ on } \partial \mathbb{B}_1\ ,
\end{equation}
has a unique solution $Kf=u\in C^{4+\alpha}(\mathbb{B}_1)\cap C^{2+\alpha}(\bar{\mathbb{B}}_1)$ and $K\in \mathcal{L}(C^{1+\alpha}(\bar{\mathbb{B}}_1), C^{4+\alpha}(\bar{\mathbb{B}}_1))$. The results so far are valid for arbitrary functions $f$ defined in $\mathbb{B}_1$ with the required regularity. We now restrict to radially symmetric functions.
Let $C_{D,rad}^{2+\alpha}(\bar{\mathbb{B}}_1)$ be the closed subspace consisting of all radially symmetric functions $w\in C^{2+\alpha}(\bar{\mathbb{B}}_1)$ such that $w= \partial_n w=0$ on $\partial \mathbb{B}_1$. The well-posedness of \eqref{hinkelstein} and the rotational invariance of the operator $B\Delta^2  -T\Delta $ ensure that $Kf$ is radially symmetric if $f$ is.  The previous properties and the Arzel\`a-Ascoli theorem thus entail that $K$ is a compact endomorphism of $C_{D,rad}^{2+\alpha}(\bar{\mathbb{B}}_1)$, and arguing as in the proof of Theorem~\ref{wildschwein} we conclude from Theorem~\ref{prrbp1} and Corollary~\ref{c22} that $K$ is strongly positive. Consequently,  we may again apply the Kre\u\i n-Rutman theorem \cite[Theorem~3.2]{A76} to complete the proof.
\end{proof}

\begin{remark}\label{idefix}
Note that we do not claim that $\mu_1$ is the principal eigenvalue of the operator $B\Delta^2 -T\Delta$ subject to homogeneous Dirichlet boundary conditions since we restrict our attention to radially symmetric functions. When $T=0$, the simplicity of the principal eigenvalue and the positivity of the corresponding normalized eigenfunction of \eqref{obelix} are already known and may be found in \cite{EGG10,GGS10}. It is worth to point out that  one does not need to restrict to the radially symmetric setting in that case since the strong sign-preserving property is true for arbitrary functions for this particular operator due to Boggio's principle \cite{Bo05}.
\end{remark}

\subsection{Radially Symmetric Stationary Solutions for a MEMS Model}\label{Sub46}

In this subsection we prove Theorem~\ref{MEMS} and thus consider the boundary value problem
\begin{equation}\label{BT}
B\Delta^2 u-T\Delta u= -\lambda g(u)\ \text{ in } \mathbb{B}_1\ ,\quad u=\partial_n u=0\ \text{ on } \partial \mathbb{B}_1\ ,
\end{equation}
where $B>0$, $T>0$, and $g\in C^2(J)$ is a non-negative and non-increasing function with $g(0) >0$. To obtain the existence of a maximal solution we use a monotonicity argument as in \cite[Chapter~11]{EGG10}. Define
\begin{equation*}
\begin{split}
\lambda_*:=\sup\{\lambda> 0\, ;\, & \eqref{BT} \text{ has a radially symmetric classical}\\
& \text{solution $u\in C^4(\mathbb{B}_1)\cap C^2(\bar{\mathbb{B}}_1)$ with $u(\bar{\mathbb{B}}_1)\subset J$}\}
\end{split}
\end{equation*}
and note that $\lambda_*>0$. Indeed, setting  for  $q\in (1,\infty)$
$$
Au:= B\Delta^2 u-T\Delta u\ ,\quad u\in W_{q,D}^4(\mathbb{B}_1):=\{u\in W_q^4(\mathbb{B}_1)\,;\, u=\partial_n u=0 \text{ on } \partial \mathbb{B}_1\}\ ,
$$ 
the operator $A: W_{q,D}^4(\mathbb{B}_1)\rightarrow L_q(\mathbb{B}_1)$ is boundedly invertible \cite[Theorem~2.20]{GGS10}, and the implicit function theorem thus implies that all zeros of $F(\lambda,u):=u+\lambda A^{-1} g(u)$ near $(\lambda,u)=(0,0)$ lie on a curve $(\lambda, U_\lambda)$ with $0\le \lambda\le \delta$ for some $\delta>0$. Since \eqref{BT} is rotationally invariant and since $q$ was arbitrary, this uniqueness property entails that $U_\lambda$ is radially symmetric and belongs to $C^4(\mathbb{B}_1)\cap C^2(\bar{\mathbb{B}}_1)$ with $U_\lambda(\bar{\mathbb{B}}_1)\subset J$, whence $\lambda_*>0$. To continue, we need the following auxiliary result:

\begin{lemma}\label{LL3}
Let $\lambda>0$ and suppose  that there is a non-positive radially symmetric classical subsolution $\sigma \in C^4(\mathbb{B}_1)\cap C^2(\bar{\mathbb{B}}_1)$ of \eqref{BT}, i.e.
$$
B\Delta^2 \sigma-T\Delta \sigma\le  -\lambda g(\sigma)\ \text{ in } \mathbb{B}_1\ ,\quad \sigma=\partial_n \sigma=0\ \text{ on } \partial \mathbb{B}_1\ ,
$$
satisfying $\sigma(\bar{\mathbb{B}}_1)\subset J$. Let $w\in C^4(\mathbb{B}_1)\cap C^2(\bar{\mathbb{B}}_1)$ be a radially symmetric function such that $\sigma\le w\le 0$ in~$\mathbb{B}_1$.
Then the boundary value problem
\begin{eqnarray}
&B\Delta^2 v-T\Delta v=-\lambda g(w)\ \text{in } \mathbb{B}_1\ ,\label{oo1}\\
& v =\partial_n v =0\ \text{ on } \partial \mathbb{B}_1\ , \label{oo2}
\end{eqnarray}
has a unique radially symmetric classical solution $v\in C^4(\mathbb{B}_1)\cap C^2(\bar{\mathbb{B}}_1)$ satisfying $\sigma\le v<0$ in $\mathbb{B}_1$.
\end{lemma}

\begin{proof}
Due to $\sigma(\bar{\mathbb{B}}_1)\subset J$ and $\sigma\le w\le 0$, the properties of $g$ and $w$ imply that $g(w)\in C^2(\bar{\mathbb{B}}_1)$. Therefore, \eqref{oo1}, \eqref{oo2} has a unique classical solution $v\in C^4(\mathbb{B}_1)\cap C^2(\bar{\mathbb{B}}_1)$ which is radially symmetric since $w$ is. Next, since $g(0)\le g(w)\le g(\sigma)$ in $\mathbb{B}_1$, which stems from the monotonicity of $g$, we also have 
$$
B\Delta^2 (\sigma-v)-T\Delta (\sigma-v)\le  -\lambda g(\sigma)+\lambda g(w)\le 0\ \text{ in } \mathbb{B}_1\ ,
$$
and Theorem~\ref{prrbp1} guarantees that $\sigma\le v$ in $\mathbb{B}_1$. Moreover, since  $g(0)>0$, a further  application of  Theorem~\ref{prrbp1} reveals that $v<0$ in $\mathbb{B}_1$, which completes the proof.
\end{proof}

Now, fix $\lambda\in (0,\lambda_*)$ and suppose there is a radially symmetric classical subsolution $\sigma \in C^4(\mathbb{B}_1)\cap C^2(\bar{\mathbb{B}}_1)$ of \eqref{BT} satisfying $\sigma(\bar{\mathbb{B}}_1)\subset J$. Since $g$ is non-negative, Theorem~\ref{prrbp1} ensures that $\sigma$ is non-positive in $\mathbb{B}_1$.
Define then a sequence $(u_n)_{n\in\mathbb{N}}$ by $u_0:=0$ and
\begin{eqnarray}
&B\Delta^2 u_n-T\Delta u_n=-\lambda g(u_{n-1})\ \text{in } \mathbb{B}_1\ ,\label{o1}\\
& u_n =\partial_n u_n =0\ \text{ on } \partial \mathbb{B}_1\ . \label{o2}
\end{eqnarray}
Since $\sigma\le u_0=0$, we may apply Lemma~\ref{LL3} repeatedly to obtain by induction that $(u_n)_{n\ge 1}$ is a well-defined sequence of radially symmetric classical solutions in $C^4(\mathbb{B}_1)\cap C^2(\bar{\mathbb{B}}_1)$ to \eqref{o1}, \eqref{o2} satisfying  $\sigma\le u_n<0$ in $\mathbb{B}_1$ for each $n\ge 1$. Let us now show by induction that $ u_n\le u_{n-1}$ in $\mathbb{B}_1$ for each $n\ge 1$. We already know that $u_1\le u_0=0$ in $\mathbb{B}_1$. Assume that $u_n\le u_{n-1}$ in $\mathbb{B}_1$ for some $n\ge 1$. Then $g(u_{n-1})\le g(u_n)$ in $\mathbb{B}_1$ so that
$$
B\Delta^2 (u_{n+1}-u_n)-T\Delta (u_{n+1}-u_n)\le  -\lambda g(u_n)+\lambda g(u_{n-1})\le 0\quad\text{in } \mathbb{B}_1\ ,
$$
whence $u_{n+1}\le u_n$ by Theorem~\ref{prrbp1}. We have thus shown that
\begin{equation}\label{p4}
\sigma\le u_{n+1}\le u_n\le u_1< 0\ \text{ in } \mathbb{B}_1\ ,\quad n\ge 1\ ,
\end{equation}
and, in particular, $0\le g(u_n)\le g(\sigma)\in L_\infty(\mathbb{B}_1)$. From \eqref{o1}, \eqref{o2} we infer that the sequence $(u_n)_{n\in\mathbb{N}}$ is bounded in $W_q^4(\mathbb{B}_1)$ for any $q\in [1,\infty)$. It is then straightforward to show that $(u_n)_{n\in\mathbb{N}}$ converges to a radially symmetric classical solution $u_\lambda\in C^4(\mathbb{B}_1)\cap C^2(\bar{\mathbb{B}}_1)$ of \eqref{BT} satisfying 
\begin{equation}\label{pp}
\sigma\le u_\lambda < 0\ \text{ in } \mathbb{B}_1
\end{equation} 
by \eqref{p4}.
Note that $u_\lambda$ does not depend on $\sigma$. Therefore, it lies above any radially symmetric classical subsolution and is in this sense a maximal solution. In particular, it is unique. To complete the existence of $u_\lambda$, it remains to construct a suitable subsolution. But, since $\lambda\in (0,\lambda_*)$, there is $\lambda_0\in (\lambda,\lambda_*)$ and a radially symmetric classical solution $u^0$ of \eqref{BT} with $\lambda$ replaced by $\lambda_0$ satisfying $u^0(\bar{\mathbb{B}}_1)\subset J$. Also, $u^0\le 0$ by Theorem~\ref{prrbp1}. Furthermore,
$$
B\Delta^2 u^0-T\Delta u^0+\lambda g(u^0)\le  (\lambda-\lambda_0) g(u^0)\le 0\quad\text{in } \mathbb{B}_1\ ,
$$
so that $u^0$ is a radially symmetric classical subsolution of \eqref{BT}. The previous analysis now ensures the existence of the maximal solution $u_\lambda$.  

Moreover, given $0<\lambda_1<\lambda_2<\lambda_*$, we deduce from \eqref{BT} and \eqref{pp} that
$$
B\Delta^2 u_{\lambda_2}-T\Delta u_{\lambda_2}+\lambda_1 g(u_{\lambda_2})=  (\lambda_1-\lambda_2) g(u_{\lambda_2})< 0\quad\text{in } \mathbb{B}_1\ ,
$$
and the maximality of $u_{\lambda_1}$ warrants that $u_{\lambda_2}\le u_{\lambda_1}$ in $\mathbb{B}_1$. We next notice that \eqref{BT}, \eqref{pp}, and the monotonicity of $g$ ensure
$$
B\Delta^2 (u_{\lambda_2}-u_{\lambda_1}) -T\Delta (u_{\lambda_2}-u_{\lambda_1})=  (\lambda_1-\lambda_2) g(u_{\lambda_1})+\lambda_2 (g(u_{\lambda_1})-g(u_{\lambda_2}))< 0\quad\text{in } \mathbb{B}_1\ .
$$
Theorem~\ref{prrbp1} then gives $u_{\lambda_2}<u_{\lambda_1}$ in $\mathbb{B}_1$.

Finally, assume that $a:=\inf J>-\infty$ and $m:=\inf_{(a,0)} g >0$. Let $\phi_1>0$ be a radially symmetric positive eigenfunction associated with the eigenvalue $\mu_1>0$ of the operator $B\Delta^2-T\Delta$ subject to homogeneous Dirichlet boundary conditions, its existence being guaranteed by Theorem~\ref{asterix}. Now, if $u$ is any radially symmetric classical solution to \eqref{BT} for some $\lambda>0$, then $a<u\le 0$ by Theorem~\ref{prrbp1} and thus
$$
a\mu_1\int_{\mathbb{B}_1} \phi_1\,\mathrm{d} x \le \mu_1\int_{\mathbb{B}_1}\phi_1 u\,\mathrm{d} x = \int_{\mathbb{B}_1}\phi_1 (B\Delta^2 u-T\Delta u)\,\mathrm{d} x= -\lambda \int_{\mathbb{B}_1}\phi_1 g(u)\,\mathrm{d} x \le -\lambda\, m \int_{\mathbb{B}_1}\phi_1 \,\mathrm{d} x\ .
$$
Therefore, $\lambda_*\le -a \mu_1/m <\infty$.
This completes the proof of Theorem~\ref{prrbp1}.

\section*{Acknowledgments}

Part of this research was done while Ph.L. was enjoying the kind hospitality of the Institut f\"ur Angewandte Mathematik of the Leibniz Universit\"at Hannover. The work of Ph.L. was partially supported by the Centre International de Math\'ematiques et d'Informatique CIMI, Toulouse.


\bibliographystyle{amsplain}
\bibliography{BP}

\end{document}